\documentclass{elsarticle}
\usepackage{amsmath}
\usepackage{amssymb}
\usepackage{amsfonts}
\usepackage{rotating}
\usepackage{graphicx}
\usepackage{floatflt,epsfig}
\usepackage{lineno,hyperref}
\usepackage{enumerate}
\usepackage{colortbl}
\usepackage{array,tabularx,tabulary,booktabs}
\usepackage{longtable}
\usepackage{multirow}
\usepackage{wrapfig}
\usepackage{subcaption}
\newcolumntype{^}{>{\currentrowstyle}}

\journal{Discrete Mathematics}
\newtheorem{theorem}{Theorem}

\newtheorem{construction}{Construction}

\begin{document}
\renewcommand{\abstractname}{Abstract}
\renewcommand{\refname}{References}
\renewcommand{\tablename}{Table.}
\renewcommand{\arraystretch}{0.9}
\sloppy

\begin{frontmatter}
\title{Construction of divisible design graphs using affine designs}

\author{Vladislav~V.~Kabanov}
\ead{vvk@imm.uran.ru}
\address{School of Mathematical Sciences, Hebei Key Laboratory of Computational Mathematics and Applications, Hebei Normal University, Shijiazhuang 050024, PR China}
\address{Krasovskii Institute of Mathematics and Mechanics, Yekaterinburg, 620077, Russia}

\begin{abstract}
A $k$-regular graph on $v$ vertices is a {\em divisible design graph} if there exist integers $\lambda_1,\lambda_2,m,n$ such that the vertex set can be partitioned into $m$ classes of size $n$ and any two different vertices from the same class have $\lambda_1$ common neighbours, and any two vertices from different classes have  $\lambda_2$ common neighbours. This paper presents two prolific constructions that produce infinite series of divisible design graphs. The first construction develops ideas of W.D.~Wallis, D.G.~Fon-Der-Flaass, and M.~Muzychuk which were created to construct new strongly regular graphs.
\end{abstract}

\begin{keyword}  Divisible design graph

\vspace{\baselineskip}
\MSC[2020] 05C51
\end{keyword}
\end{frontmatter}

\section{Introduction}

A $k$-regular graph on $v$ vertices is a {\em divisible design graph} if there exist integers $\lambda_1,\lambda_2,m,n$ such that the vertex set can be partitioned into $m$ classes of size $n$ and any two different vertices from the same class have $\lambda_1$ common neighbours, and any two vertices from different classes have  $\lambda_2$ common neighbours.
The partition of a divisible design graph with parameters $(v,k,\lambda_1,\lambda_2,m,n)$ into  classes is called  a {\em canonical partition}.
If $m=1$, $n=1$, or $\lambda_1=\lambda_2=\lambda$, then a divisible design graph is a strongly regular with parameters $(v,k,\lambda,\lambda)$. If this is not the case, then a divisible design graph is called {\em proper}. The canonical partition of a proper divisible design graph yields a partition of its adjacency matrix. Thus, the adjacency matrix of any proper divisible design graph consists of parts called blocks according to the canonical partition.

Note that a divisible design graph is a divisible design if the vertices are considered as points, and the neighbourhoods of the vertices as blocks. R.C.~Bose and W.S.~Connor studied the combinatorial properties of divisible designs in~\cite{BC}. 

Divisible design graphs were first provided by W.H.~Haemers, H.~Kharaghani and M.~Meulenberg in~\cite{HKM}.  In particular, the authors have proposed nineteen constructions of divisible design graphs using various combinatorial structures
and two sporadic examples. Some new combinatorial constructions of divisible design graphs were also provided in~\cite{CH},\cite{VK}, \cite{PSh},\cite{Sh} and constructions as Cayley graphs in \cite{CSv}, \cite{KSh}. 

W.D.~Wallis proposed in~\cite{W} a new construction of strongly regular graphs using affine designs and a Steiner 2-design. Later D.G.~Fon-Der-Flaass found how to modify a partial case of Wallis construction, when the corresponding Steiner $2$-design has blocks of size $2$, in order to obtain hyperexponentially many strongly regular graphs with the same parameters \cite{FF}.
M.~Muzychuk in~\cite{MM} showed how to modify Fon-Der-Flaass ideas in order to cover all the cases of Wallis construction. Moreover, he showed that a Steiner $2$-design in the original Wallis construction can be replaced by a partial linear space and  discovered new prolific constructions of strongly regular graphs. 

This paper presents two prolific constructions that produce infinite series of divisible design graphs (Section \ref{FirstCon} and \ref{ParCom}). The first construction develops ideas of W.D.~Wallis, D.G.~Fon-Der-Flaass, and M.~Muzychuk. According to references, these constructions are new. In some cases, the parameters of the series coincide with those of the divisible design graphs found in \cite{BG} (see Section \ref{BhGor}). The smallest new example is provided in Section \ref{SmEx}. 

{\em  An affine design} $\mathcal{D}=(\mathcal{P}, \mathcal{B})$  with parameters $q$ and $r$ is a design, where $\mathcal{P}$ is a set of points and $\mathcal{B}$ is a set of blocks,  with the  following two properties: every two blocks are either disjoint or intersect in $r$  points; each block together with all blocks disjoint from it forms a parallel class: a set of $q$ mutually disjoint blocks partitioning all points of the design.

Any $d$-dimensional affine space over a finite field of order $q$  is point-hyperplane design with $r=q^{d-2}$. This design has $q^2 r = q^d$ points, any block contains $q r = q^{d-1}$ points, the number of  parallel classes is $(q^2 r-1)/(q-1) = (q^d-1)/(q-1)$, and the number of blocks containing any two distinct points is 
$(q r -1)/(q-1) = (q^{d-1}-1)/(q-1)$. Other known examples of affine designs are finite Hadamard $3$-designs $(q=2)$, Desarguesian and non-Desarguesian finite affine planes.

We only consider undirected graphs without loops or multiple edges. For all necessary information about graphs we refer to \cite{BH} and about designs to \cite{CDW}.

\section{Sporadic example}\label{se}

The following block matrix is the adjacency matrix of a divisible design graph with parameters 
$(28,6,2,1,7,4)$ which was found by D.I.~Panasenko and L.V.~Shalaginov in \cite[Construction 23]{PSh} using computer calculations. 

$$M = \begin{bmatrix}
O    & a & b & c & O &   O &   O\\
a^T  & d & O & O & e &   O &   O\\
b^T  & O & d & O & O &   e &   O\\
c^T  & O & O & d & O &   O &   e\\
O    & e^T & O & O & O & a &   c\\
O    & O & e^T & O & a^T & O &   f\\
O    & O & O & e^T & c^T & f^T & O
\end{bmatrix},$$
where $O$ is the zero $(4\times 4)$-matrix, and
$$a = \begin{bmatrix}
1 & 1 & 0 & 0\\
1 & 1 & 0 & 0\\
0 & 0 & 1 & 1\\
0 & 0 & 1 & 1
\end{bmatrix},\quad
b = \begin{bmatrix}
1 & 1 & 0 & 0\\
0 & 0 & 1 & 1\\
1 & 1 & 0 & 0\\
0 & 0 & 1 & 1
\end{bmatrix},\quad
c = \begin{bmatrix}
1 & 1 & 0 & 0\\
0 & 0 & 1 & 1\\
0 & 0 & 1 & 1\\
1 & 1 & 0 & 0
\end{bmatrix},$$
$$d = \begin{bmatrix}
0 & 1 & 1 & 0\\
1 & 0 & 0 & 1\\
1 & 0 & 0 & 1\\
0 & 1 & 1 & 0
\end{bmatrix},\quad
e = \begin{bmatrix}
1 & 0 & 1 & 0\\
0 & 1 & 0 & 1\\
1 & 0 & 1 & 0\\
0 & 1 & 0 & 1
\end{bmatrix},\quad
f = \begin{bmatrix}
1 & 0 & 0 & 1\\
0 & 1 & 1 & 0\\
0 & 1 & 1 & 0\\
1 & 0 & 0 & 1
\end{bmatrix}.$$

It is easy to see that if we replace in $M$ any zero block with $0$ and any nonzero block with $1$, we obtain the incidence matrix of the Fano plane. 

$$A=\begin{bmatrix}
0 & 1 & 1 & 1 & 0 & 0 & 0\\
1 & 1 & 0 & 0 & 1 & 0 & 0\\
1 & 0 & 1 & 0 & 0 & 1 & 0\\
1 & 0 & 0 & 1 & 0 & 0 & 1\\
0 & 1 & 0 & 0 & 0 & 1 & 1\\
0 & 0 & 1 & 0 & 1 & 0 & 1\\
0 & 0 & 0 & 1 & 1 & 1 & 0
\end{bmatrix}.$$

The Fano plane can be considered as a symmetric $2$-$(7,3,1)$-design. This observation led to the construction of divisible design graphs described in the next section.

\section{First construction}\label{FirstCon}

The construction presented in this section is a generalization of Construction 1 and Construction 3 from \cite{VK}.
\smallskip

Let $\mathcal{D}_1, \dots ,\mathcal{D}_m$ be affine designs with parameters 
$(q,r)$ which are not necessarily isomorphic, and $\kappa=(q^2 r - 1)/(q-1)$ be the number of parallel classes of blocks for each $\mathcal{D}_i$.
 Let $\mathcal{D}_i=(\mathcal{P}_i, \mathcal{B}_i)$ and the parallel classes in each $\mathcal{D}_i$ are enumerated by integers from $[\kappa]$.  
We denote the $j$-th parallel class of $\mathcal{D}_i$ by $\mathcal{B}_i^j$
and the block in the parallel class $\mathcal{B}_i^j$ that contains $x$ by $B_i^j(x)$.
\smallskip

Let $A$ be a symmetric $m\times m$ incidence matrix of a symmetric $2$-design with parameters $(m,\kappa,\lambda)$.
Let $L=(e(i,j))$ be an $(m\times m)$-matrix which is obtained from
$A$ by changing $\kappa$ nonzero entries in each rows of $A$ by all integers
from $[\kappa]=\{1,2,\dots,\kappa\}$. It means that the map from $[\kappa]$ to the set of nonzero entries for each row of $L$ is a bijection.
\smallskip

For each pair $i, j$ for which $e(i,j)\neq 0$, select an arbitrary bijection
$$\sigma_{i,j} : \mathcal{B}_i^{e(i,j)} \rightarrow \mathcal{B}_j^{e(j,i)}.$$ 
This definition is correct because $A$ is a symmetric $(m\times m)$-matrix so $e(i,j)=0$
if and only if $e(j,i)=0$.

We require that $\sigma_{i,j}=\sigma_{j,i}^{-1}$ for $i\neq j$ and $\sigma_{i,i}$ is the identical map on
$\mathcal{B}_i^{e(i,i)}$.

\begin{construction}\label{Con1}
Let $\Gamma$ be a graph defined as follows:
\begin{itemize}
    \item The vertex set of $\Gamma$ is  $\displaystyle V(\Gamma)=\bigcup_{i=1}^{m} \mathcal{P}_i.$ 
    \item Two different vertices $x\in \mathcal{P}_i$ and $y\in \mathcal{P}_j$ are adjacent in $\Gamma$ if and only if $e(i,j)\neq 0$, and 
    $$y \notin \sigma_{i,j}(B_i^{e(i,j)}(x))\quad \mathit{for} \quad  i,j\in [m].$$ 
\end{itemize}
\end{construction}

\begin{theorem}\label{Th1}  
If $\Gamma$ is a graph from Construction~\ref{Con1}, then  $\Gamma$ is a divisible design graph with parameters  
 $$v = q^2 r m, \qquad k = q r (q^2 r - 1),$$
   $$\lambda_1 = q r (q^2 r - q r - 1),
   \qquad \lambda_2 = \lambda r (q-1)^2 ,$$
 $$ m,  \qquad n=q^2 r.$$

\end{theorem}
\begin{proof} The proof is very similar to the proof of \cite[Theorem 1]{VK}, but at the same time there are important differences, since the more complicated structure of $L$ must be taken into account. 

If $\Gamma$ is a graph from Construction~\ref{Con1}, then the number of vertices of $\Gamma$ is equal to $q^2 r m$.

Let $x$ be a vertex of $\Gamma$ belonging to $\mathcal{P}_i$. There are exactly $\kappa$ nonzero entries in each row of $L$, thus
 $$\Gamma(x)= \bigcup_{j=1}^{\kappa} (\mathcal{P}_j\setminus \sigma_{i,j}(B_i^{e(i,j)}(x))).$$
Clearly, $$|\mathcal{P}_j\setminus \sigma_{i,j}(B_i^{e(i,j)}(x)|=
|\mathcal{P}_j|-|\sigma_{i,j}(B_i^{e(i,j)}(x)|= q^2 r - q r.$$
Hence, $\Gamma$ is a regular graph of degree 
$k = \kappa q r (q - 1) = q r (q^2 r - 1).$

Let $x$ and $y$ be two different vertices in $\Gamma$ and both belonging to $\mathcal{P}_i$. There are exactly $(q r -1)/(q-1)$ blocks in $\mathcal{D}_i$ that contain both $x$ and $y$.
Thus, there are $(q r -1)/(q-1)$ cases for $\mathcal{P}_j$ in which $\sigma_{i,j}(B_i^{e(i,j)}(x))$ and $\sigma_{i,j}(B_i^{e(i,j)}(y))$ are the same blocks. Hence, $x$ and $y$ have exactly $q^2 r - q r$ common neighbours in each of these cases.
In the remaining possible cases $\sigma_{i,j}(B_i^{e(i,j)}(x))$ and $\sigma_{i,j}(B_i^{e(i,j)}(y))$  are disjoint blocks. Hence, $x$ and $y$ have exactly 
$q^2 r - 2q r$ common neighbours in each of the remaining possible cases. 
Therefore, the number of common neighbours for $x$ and $y$ equals 
$$(q^2 r - q r) \frac{(q r -1)}{(q-1)} + (q^2 r - 2q r)\left(\frac{(q^2 r-1)}{(q-1)} - \frac{(q r-1)}{(q-1)}\right) = $$ $$ = q r (q^2 r - q r - 1).$$

Let $x$ be in $\mathcal{P}_i$, and $y$ be in $\mathcal{P}_j$, where $i\neq j$. 
There are exactly $\lambda$ cases when $e(i,h)$ and $e(j,h)$ both are nonzero for $h\in [m]$. Additionally, $e(h,i)$ and $e(h,j)$ are different by the definition of $L$. For each of these cases, there are two blocks
$$\sigma_{ih}(B_i^{e(i,h)}(x))\quad \mathrm{and}\quad \sigma_{jh}(B_j^{e(j,h)}(y))$$ in $\mathcal{D}_h$ that are not parallel and therefore have exactly $r$ common points.
So $x$ and $y$ have exactly $q^2 r - 2q r + r$ common neighbours in each of these $\lambda$ cases. Hence, the number of common neighbours for $x$ and $y$ equals 
$$\lambda (q^2 r - 2q r + r) = \lambda r (q-1)^2.$$ 
\end{proof}\hfill $\square$
\medskip

Note that the equality $\kappa(\kappa -1)=(m-1)\lambda$ holds for any symmetric $2$-design with parameters $(m,\kappa,\lambda)$, including the two trivial cases $m=\kappa$ and $m=\kappa+1$.
Therefore, if $m=\kappa$, then $\lambda =\kappa$. Thus, we have the situation from \cite[Theorem 1]{VK} and $\Gamma$ has parameters   
$$v = q^2 r (q^2 r - 1)/(q-1),\quad k = q r (q^2 r - 1),$$
$$\lambda_1 = q  r (q^2 r - q r - 1),\quad \lambda_2 = r (q-1)(q^2 r - 1),$$ $$m = (q^2 r - 1)/(q-1),\quad n = q^2 r.$$ If $m=\kappa+1$, then $\lambda =\kappa -1$. Thus, we have the situation from \cite[Theorem 3]{VK} and $\Gamma$ has parameters    
$$v = q^2 r(q^2 r +q-2)/(q-1),\quad k = q r (q^2 r - 1),$$
$$\lambda_1 = q r (q^2 r - q r - 1),\quad \lambda_2 = q r (q - 1)(q r - 1),$$ 
$$m = (q^2 r +q-2)/(q-1),\quad n = q^2 r.$$

\section{Partial complement}\label{ParCom}

In general, the complement of a divisible design graph is not a divisible design graph again. In \cite{HKM}, the authors considered the complement of the off-diagonal blocks of the adjacency block matrix of a divisible design graph. There are some cases when the result of such a partial complement is a new divisible design graph (see Proposition 4.15 and Proposition 4.16 \cite{HKM}). We consider another variant of partial complement of divisible design graphs.

\begin{construction}\label{Con2}
Let $\Gamma$ be a graph from Construction \ref{Con1} with parameters 
 $$v = q^2 r m, \qquad k = q r (q^2 r - 1),$$
   $$\lambda_1 = q r (q^2 r - q r - 1), \qquad \lambda_2 = \lambda r (q-1)^2 ,$$
 $$ m,  \qquad n=q^2 r,$$
and $M$ be the canonical block matrix of $\Gamma$.
Assume that every diagonal block of $M$ is nonzero.
Replace all zero off-diagonal blocks with all-ones matrices of the same size.
Denote the graph with this new matrix by $\Gamma^\ast$.
\end{construction}

\begin{theorem}\label{Th2}  
$\Gamma^\ast$ is a divisible design graph with parameters  
$$v^\ast = q^2 r m,\quad k^\ast = q r (q^2 r - 1) + q^2 r (m-\kappa),$$
$$\lambda_1^\ast = q r (q^2 r - q r - 1) + q^2 r (m-\kappa),$$
$$\lambda_2^\ast = r (q-1)^2\lambda + 
2(q^2 r - q r)(\kappa-\lambda) + q^2 r (m-2\kappa+\lambda),$$ $$m^\ast =m,\quad n^\ast = q^2 r.$$
\end{theorem} 

\begin{proof}
It is obvious that $\Gamma$ and $\Gamma^\ast$ have the same number of vertices. Since all diagonal blocks of $M$ are nonzero the  number of non-diagonal zero blocks in each  row of the adjacency block matrix of $\Gamma$ equals $m-\kappa$. This gives $q^2 r (m-\kappa)$ as an increase for $k^\ast$ and $\lambda_1^\ast$. The parameter $\lambda_2^\ast$ increases by $2(q^2 r -2q r)(\kappa-\lambda)$ where zero blocks meet non-zero blocks, and by $q^2 r (m-2\kappa+\lambda)$ where non-zero blocks meet.
Clearly, $\Gamma$ and $\Gamma^\ast$ have the same number and size of their canonical classes.
\end{proof} \hfill $\square$

\section{Divisible design graphs from symplectic graphs over rings}\label{BhGor}

In \cite{BG}, A. Bhowmik and S. Goryainov provided a new construction of divisible design graphs from symplectic graphs over rings. Let $q$ be a prime power, $\mathbb{F}_q$ be a finite field of order $q$, $K$ be a finite commutative ring with identity having exactly one non-trivial ideal $J$ such that $K/J\cong F_q$, and $K^{\times}$ be the set of units of $K$. For any integer $e \geq 2$,
let $V'$ be defined as follows: $$\{(a_1,a_2,\dots,a_{2e})\, :\, a_1,a_2,\dots,a_{2e}\in K,\,\, a_j\in K^{\times}\, \mathrm{for\, some}\, j\in [2e]\}.$$ 
For $a=(a_1,a_2,\dots,a_{2e})$ and $b=(b_1,b_2,\dots, b_{2e})$, let the equivalence relation $\sim$ on $V'$ be defined as: 
 $a\sim b$ iff there exists $\lambda\in K^{\times}$ such that $a_j =\lambda b_j$ for all $j\in 
 [2e].$
Let $V$ denote the set of equivalence classes on $V'$ corresponding to this equivalence relation, and $[a_1,a_2,\dots,a_{2e}]$ be the equivalence class of $(a_1,a_2,\dots,a_{2e})$. Let 
$$M =\left( \begin{array}{cc} O_e & I_e \\ -I_e & O_e \end{array}\right),$$ where $I_e$ is the identity $(e\times e)$-matrix and $O_e$ is the zero $(e\times e)$-matrix.

The authors in \cite{BG} define two graphs $X(2e,K)$ and $Y(2e,K)$, both with vertex set 
$V$ :
\begin{itemize}
    \item $[a_1,a_2,\dots,a_{2e}]$ and $[b_1,b_2,\dots, b_{2e}]$ is adjacent in $X(2e,K)$ iff
$$(a_1,a_2,\dots,a_{2e})M(b_1,b_2,\dots, b_{2e})^T\in K\setminus \{0\};$$
    \item $[a_1,a_2,\dots,a_{2e}]$ and $[b_1,b_2,\dots, b_{2e}]$ is adjacent in $Y(2e,K)$ iff
$$(a_1,a_2,\dots,a_{2e})M(b_1,b_2,\dots, b_{2e})^T\in J\setminus \{0\}.$$
\end{itemize}

By \cite[Theorem 1]{BG} $X(2e,K)$ is a divisible design graph with parameters 
$$v=q^{2e-1}(q^{2e} - 1)/(q-1),\quad k=q^{4e-2}+q^{4e-3}-q^{2e-2},$$
$$\lambda_1=q^{4e-2}+q^{4e-3}-q^{4e-4}-q^{2e-2},\quad 
\lambda_2=q^{4e-2}+q^{4e-3}-q^{4e-4}-q^{4e-5}-q^{2e-2}+q^{2e-3},$$
$$m=(q^{2e} - 1)/(q-1),\quad n=q^{2e-1}.$$
By \cite[Theorem 2]{BG} $Y(2e,K)$ is a divisible design graph with parameters 
$$v=q^{2e-1}(q^{2e}-1)/(q-1),\quad k=q^{4e-3}-q^{2e-2},\quad
\lambda_1=q^{4e-3}-q^{4e-4}-q^{2e-2},$$
$$\lambda_2=q^{4e-4}-q^{4e-5}-q^{2e-2}+q^{2e-3},\quad
m=(q^{2e} - 1)/(q-1),\quad n=q^{2e-1}.$$

By the definition $Y(2e,K)$ is a spanning subgraph of $X(2e,K)$. Therefore, it is not difficult to see that $X(2e,K)$ is a partial complement of $Y(2e,K)$. Moreover, if $d=2e-1$ and $A$ from Construction \ref{Con1} is the incident matrix of a symmetric $2$-design with parameters 
$((q^{d+1}-1)/(q-1),(q^d-1)/(q-1),(q^{d-1}-1)/(q-1))$, then the parameters of $Y(2e,K)$ coincide with the parameters of divisible design graphs from Theorem \ref{Th1}. To verify isomorphism $Y(2e,K)$ to some 
of divisible design graphs from Construction \ref{Con1} it is enough to find appropriate bijections
$\sigma_{i,j}$. We checked the isomorphism for the graphs $Y(4,K_1)$ and $Y(4,K_2)$, where $K_1=\mathbb{Z}/p^2\mathbb{Z}$ and $K_2=\mathbb{F}_p[x]/\langle x^2\rangle$ for $p\in \{2,3\}$ using the GRAPE package \cite{GRAPE} for the GAP system \cite{GAP}.

\section{Smallest new divisible design graph from Construction 1}\label{SmEx}

 D.I.~Panasenko and L.V.~Shalaginov in \cite{PSh} found all proper divisible design graphs with the number of vertices no more than $39$, with the exception of three tuples of parameters: $(32, 15, 6, 7, 4, 8)$, $(32, 17, 8, 9, 4, 8)$ and $(36, 24, 15, 16, 4, 9)$, using computer calculations. All divisible design graphs with parameters $(36, 24, 15, 16, 4, 9)$ were found in \cite{GK}. 
 The smallest example of a divisible design graph from Construction 1 has parameters $(12,6,2,3,3,4)$. 
This divisible design graph was known from Construction 4.20 [HKM] as the line graph of the octahedron.
The second example is found in \cite[Construction 23]{PSh} and described in section \ref{se}. There are no other examples from Construction 1  and there are no examples from Construction 2 in the list from \cite{PSh}. 

 The Fano plane is the unique nontrivial symmetric $2$-design with $\kappa =3$. The complement of $2$-$(7,3,1)$-design is $2$-$(7,4,2)$-design. 
 We can use this symmetric 2-design and the $2$-dimensional point-hyperplane affine designs over the finite field of order $3$ (which has exactly $4$ parallel classes of blocks) to construct
a divisible design graph with  parameters $(63,24,15,8,7,9)$. This example cannot be obtained from the constructions in \cite{BG}.

Choose $L$ as follows:
$$\begin{bmatrix}
1 & 0 & 0 & 0 & 2 & 3 & 4\\
0 & 0 & 1 & 2 & 0 & 3 & 4\\
0 & 1 & 0 & 2 & 3 & 0 & 4\\
0 & 1 & 2 & 0 & 3 & 4 & 0\\
1 & 0 & 2 & 3 & 4 & 0 & 0\\
1 & 2 & 0 & 3 & 0 & 4 & 0\\
1 & 2 & 3 & 0 & 0 & 0 & 4
\end{bmatrix}.$$ 
Then $$M=\begin{bmatrix}
m_{11} & O & O & O & m_{12} & m_{13} & m_{14} \\
O & O & m_{11} & m_{12} & O & m_{23} & m_{24} \\
O & m_{11} & O & m_{22} & m_{23} & O & m_{34} \\
O & m_{21} & m_{22} & O & m_{33} & m_{34} & O \\
m_{21} & O & m_{32} & m_{33} & m_{44} & O & O \\
m_{31} & m_{32} & O & m_{43} & O & m_{44} & O \\
m_{41} & m_{42} & m_{43} & O & O & O & m_{44}
\end{bmatrix}$$ 
is the adjacency matrix of a divisible design graph with parameters $(63,24,15,8,7,9)$, where $O$ is the zero $(9\times 9)$-matrix, and
\begin{center}
{\footnotesize
$m_{11}=\begin{bmatrix}
0 & 0 & 0 & 1 & 1 & 1 & 1 & 1 & 1 \\
0 & 0 & 0 & 1 & 1 & 1 & 1 & 1 & 1 \\
0 & 0 & 0 & 1 & 1 & 1 & 1 & 1 & 1 \\
1 & 1 & 1 & 0 & 0 & 0 & 1 & 1 & 1 \\
1 & 1 & 1 & 0 & 0 & 0 & 1 & 1 & 1 \\
1 & 1 & 1 & 0 & 0 & 0 & 1 & 1 & 1 \\
1 & 1 & 1 & 1 & 1 & 1 & 0 & 0 & 0 \\
1 & 1 & 1 & 1 & 1 & 1 & 0 & 0 & 0 \\
1 & 1 & 1 & 1 & 1 & 1 & 0 & 0 & 0
\end{bmatrix},$
 $m_{12}=\begin{bmatrix}
0 & 0 & 0 & 1 & 1 & 1 & 1 & 1 & 1 \\
1 & 1 & 1 & 0 & 0 & 0 & 1 & 1 & 1 \\
1 & 1 & 1 & 1 & 1 & 1 & 0 & 0 & 0 \\
0 & 0 & 0 & 1 & 1 & 1 & 1 & 1 & 1 \\
1 & 1 & 1 & 0 & 0 & 0 & 1 & 1 & 1 \\
1 & 1 & 1 & 1 & 1 & 1 & 0 & 0 & 0 \\
0 & 0 & 0 & 1 & 1 & 1 & 1 & 1 & 1 \\
1 & 1 & 1 & 0 & 0 & 0 & 1 & 1 & 1 \\
1 & 1 & 1 & 1 & 1 & 1 & 0 & 0 & 0
\end{bmatrix},$
 $m_{13}=\begin{bmatrix}
0 & 0 & 0 & 1 & 1 & 1 & 1 & 1 & 1 \\
1 & 1 & 1 & 0 & 0 & 0 & 1 & 1 & 1 \\
1 & 1 & 1 & 1 & 1 & 1 & 0 & 0 & 0 \\
1 & 1 & 1 & 0 & 0 & 0 & 1 & 1 & 1 \\
1 & 1 & 1 & 1 & 1 & 1 & 0 & 0 & 0 \\
0 & 0 & 0 & 1 & 1 & 1 & 1 & 1 & 1 \\
1 & 1 & 1 & 1 & 1 & 1 & 0 & 0 & 0 \\
0 & 0 & 0 & 1 & 1 & 1 & 1 & 1 & 1 \\
1 & 1 & 1 & 0 & 0 & 0 & 1 & 1 & 1
\end{bmatrix},$
$m_{14}=\begin{bmatrix}
0 & 0 & 0 & 1 & 1 & 1 & 1 & 1 & 1 \\
1 & 1 & 1 & 0 & 0 & 0 & 1 & 1 & 1 \\
1 & 1 & 1 & 1 & 1 & 1 & 0 & 0 & 0 \\
1 & 1 & 1 & 1 & 1 & 1 & 0 & 0 & 0 \\
0 & 0 & 0 & 1 & 1 & 1 & 1 & 1 & 1 \\
1 & 1 & 1 & 0 & 0 & 0 & 1 & 1 & 1 \\
1 & 1 & 1 & 0 & 0 & 0 & 1 & 1 & 1 \\
1 & 1 & 1 & 1 & 1 & 1 & 0 & 0 & 0 \\
0 & 0 & 0 & 1 & 1 & 1 & 1 & 1 & 1
\end{bmatrix},$
 $m_{22}=\begin{bmatrix}
0 & 1 & 1 & 0 & 1 & 1 & 0 & 1 & 1 \\
1 & 0 & 1 & 1 & 0 & 1 & 1 & 0 & 1 \\
1 & 1 & 0 & 1 & 1 & 0 & 1 & 1 & 0 \\
0 & 1 & 1 & 0 & 1 & 1 & 0 & 1 & 1 \\
1 & 0 & 1 & 1 & 0 & 1 & 1 & 0 & 1 \\
1 & 1 & 0 & 1 & 1 & 0 & 1 & 1 & 0 \\
0 & 1 & 1 & 0 & 1 & 1 & 0 & 1 & 1 \\
1 & 0 & 1 & 1 & 0 & 1 & 1 & 0 & 1 \\
1 & 1 & 0 & 1 & 1 & 0 & 1 & 1 & 0
\end{bmatrix},$
 $m_{23}=\begin{bmatrix}
0 & 1 & 1 & 0 & 1 & 1 & 0 & 1 & 1 \\
1 & 0 & 1 & 1 & 0 & 1 & 1 & 0 & 1 \\
1 & 1 & 0 & 1 & 1 & 0 & 1 & 1 & 0 \\
1 & 0 & 1 & 1 & 0 & 1 & 1 & 0 & 1 \\
1 & 1 & 0 & 1 & 1 & 0 & 1 & 1 & 0 \\
0 & 1 & 1 & 0 & 1 & 1 & 0 & 1 & 1 \\
1 & 1 & 0 & 1 & 1 & 0 & 1 & 1 & 0 \\
0 & 1 & 1 & 0 & 1 & 1 & 0 & 1 & 1 \\
1 & 0 & 1 & 1 & 0 & 1 & 1 & 0 & 1
\end{bmatrix},$
$m_{24}=\begin{bmatrix}
0 & 1 & 1 & 0 & 1 & 1 & 0 & 1 & 1 \\
1 & 0 & 1 & 1 & 0 & 1 & 1 & 0 & 1 \\
1 & 1 & 0 & 1 & 1 & 0 & 1 & 1 & 0 \\
1 & 1 & 0 & 1 & 1 & 0 & 1 & 1 & 0 \\
0 & 1 & 1 & 0 & 1 & 1 & 0 & 1 & 1 \\
1 & 0 & 1 & 1 & 0 & 1 & 1 & 0 & 1 \\
1 & 0 & 1 & 1 & 0 & 1 & 1 & 0 & 1 \\
1 & 1 & 0 & 1 & 1 & 0 & 1 & 1 & 0 \\
0 & 1 & 1 & 0 & 1 & 1 & 0 & 1 & 1
\end{bmatrix},$
$m_{33}=\begin{bmatrix}
0 & 1 & 1 & 1 & 1 & 0 & 1 & 0 & 1 \\
1 & 0 & 1 & 0 & 1 & 1 & 1 & 1 & 0 \\
1 & 1 & 0 & 1 & 0 & 1 & 0 & 1 & 1 \\
1 & 0 & 1 & 0 & 1 & 1 & 1 & 1 & 0 \\
1 & 1 & 0 & 1 & 0 & 1 & 0 & 1 & 1 \\
0 & 1 & 1 & 1 & 1 & 0 & 1 & 0 & 1 \\
1 & 1 & 0 & 1 & 0 & 1 & 0 & 1 & 1 \\
0 & 1 & 1 & 1 & 1 & 0 & 1 & 0 & 1 \\
1 & 0 & 1 & 0 & 1 & 1 & 1 & 1 & 0
\end{bmatrix},$
$m_{34}=\begin{bmatrix}
0 & 1 & 1 & 1 & 1 & 0 & 1 & 0 & 1 \\
1 & 0 & 1 & 0 & 1 & 1 & 1 & 1 & 0 \\
1 & 1 & 0 & 1 & 0 & 1 & 0 & 1 & 1 \\
1 & 1 & 0 & 1 & 0 & 1 & 0 & 1 & 1 \\
0 & 1 & 1 & 1 & 1 & 0 & 1 & 0 & 1 \\
1 & 0 & 1 & 0 & 1 & 1 & 1 & 1 & 0 \\
1 & 0 & 1 & 0 & 1 & 1 & 1 & 1 & 0 \\
1 & 1 & 0 & 1 & 0 & 1 & 0 & 1 & 1 \\
0 & 1 & 1 & 1 & 1 & 0 & 1 & 0 & 1
\end{bmatrix}.$
$m_{44}=\begin{bmatrix}
0 & 1 & 1 & 1 & 0 & 1 & 1 & 1 & 0\\
1 & 0 & 1 & 1 & 1 & 0 & 0 & 1 & 1\\
1 & 1 & 0 & 0 & 1 & 1 & 1 & 0 & 1\\
1 & 1 & 0 & 0 & 1 & 1 & 1 & 0 & 1\\
0 & 1 & 1 & 1 & 0 & 1 & 1 & 1 & 0\\
1 & 0 & 1 & 1 & 1 & 0 & 0 & 1 & 1\\
1 & 0 & 1 & 1 & 1 & 0 & 0 & 1 & 1\\
1 & 1 & 0 & 0 & 1 & 1 & 1 & 0 & 1
\end{bmatrix}.$}
\end{center}
and $m_{ji}=m_{ij}^T$. Since there are  zeros on the main diagonal of $L$, it is impossible to apply Construction 2 in this case. The smallest known example of Construction 2 has parameters $(120,92,76,70,15,8)$, which was first found in \cite{BG}.

\section{Concluding remarks}

In Section \ref{FirstCon}, we provided construction of divisible design graphs, which depends on choosing a matrix $L$ and bijections $\sigma_{i,j}$. We have not estimated the number of isomorphism classes for divisible design graphs that can be obtained from the constructions, but this can  be apparently done in the same way as in \cite[Proposition 3.5]{MM}. The question of mutual isomorphisms of these divisible design graphs is really complicated and requires further study. 
The other question is whether any graph found in \cite{BG} can be obtained from our constructions by choosing a suitable matrix $L$ and bijections $\sigma_{i,j}$.

\section*{Acknowledgements}  

The author was supported by the grant of The Natural Science Foundation of Hebei Province (project No.~A2023205045).

The author is really grateful to the anonymous referees for their comments and corrections that helped improve the paper.


\begin{thebibliography}{99}

\bibitem{BG} A. Bhowmik, S. Goryainov, Divisible design graphs from symplectic graphs over rings with precisely three ideals,  arXiv:2412.04962v1 [math.CO] 
\url{https://doi.org/10.48550/arXiv.2412.04962}

\bibitem{BC} R.C.~Bose, W.S.~Connor, Combinatorial properties of group divisible incomplete block designs, {\it Ann. Math. Statist.}, 23 (1952) 367--383.
\url{https://doi.org/10.1214/AOMS/1177729382}

\bibitem{BH} A.E.~Brouwer, W.H.~Haemers, Spectra of graphs, in: Springer Universitext, 2012.

\bibitem{CDW} Ch.J.~Colbourn, J.H.~Dinitz, I.M.~Wanless, Handbook of Combinatorial Designs (2nd ed.) 2007.

\bibitem{CH} D.~Crnkovi\'c, W.~H.~Haemers, Walk-regular divisible design graphs, {\it Designs, Codes and Cryptography}, 72 (2014) 165--175. 
\url{https://doi.org/10.1007/s10623-013-9861-0}

\bibitem{CSv} D. Crnkovi\'c, A. Švob, New constructions of divisible design Cayley graphs, Graphs Comb., 38 (2022) Article number 17. \url{https://doi.org/10.1007/s00373-021-02440-4}

\bibitem{FF} D.G.~Fon-Der-Flaass, New prolific constructions of strongly regular graphs,{\it Adv. Geom.}, 2 (2002), 301--306.

\bibitem{GK} A.L.~Gavrilyuk,  V.V.~Kabanov, Strongly regular graphs decomposable into a divisible design graph and a Hoffman coclique. {\it  Des. Codes Cryptogr.} 92(5),  (2024) 1379--1391.

\bibitem{GAP}  The GAP Group, GAP -- Groups, Algorithms, and Programming, Version 4.11.1; 2021. \url{https://www.gap-system.org}

\bibitem{HKM} W.H.~Haemers, H.~Kharaghani, M.~Meulenberg, Divisible design graphs, {\it J. Combinatorial Theory, Series A}, 118 (2011) 978--992. 
\url{https://doi.org/10.1016/j.jcta.2010.10.003}

\bibitem{VK} V.V. Kabanov, New versions of the Wallis-Fon-Der-Flaass construction to create divisible design graphs, {\it  Discrete Mathematics}, 345(11) (2022) Article ID 113054. \url{https://doi.org/10.1016/j.disc.2022.113054}

\bibitem{KSh} V.V.~Kabanov, L.V.~Shalaginov, On divisible design Cayley graphs, {\it The Art of Discrete and Applied Mathematics}, 4 (2021) \#P2.02. \url{https://doi.org/10.26493/2590-9770.1340.364}

\bibitem{MM} M.~Muzychuk, A generalization of Wallis-Fon-Der-Flaass construction of strongly regular graphs, {\it J. Algebr. Comb.}, 25 (2007) 169--187.
\url{https://doi.org/10.1007/s10801-006-0030-7}

\bibitem{PSh} D.I.~Panasenko, L.V.~Shalaginov, Classification of divisible design graphs with at most 39 vertices, {\it J. Comb. Des.}, 30(4) (2022) 205--219. 
\url{https://doi.org/10.1002/jcd.21818}

\bibitem{Sh} L.V.~Shalaginov, Divisible design graphs with parameters (4n, n+2, n-2, 2, 4, n) and (4n, 3n-2, 3n-6, 2n-2, 4, n), {\it Siberian Electronic Mathematical Reports}, 18(2) (2021) 
1742--1756. 
\url{https://doi.org/10.33048/semi.2021.18.134}

\bibitem{GRAPE} L.H.~Soicher, The GRAPE package for GAP, Version 4.8.5, 2021. \url{https://gap-packages.github.io/grape}

\bibitem{W} W.D.~Wallis, Construction of strongly regular graphs using affine designs, {\it Bull. Austral. Math. Soc.}, 4 (1971), 41--49.
\url{https://doi.org/10.1017/S0004972700046244}

\bibitem{DezaGraphs} Strictly Deza graphs \url{ http://alg.imm.uran.ru/dezagraphs/ddg.php }

\end{thebibliography}
\end{document}